\documentclass{article}
\usepackage{amsfonts}
\usepackage{amsmath}
\usepackage[mathscr]{eucal}
\usepackage{amssymb}
\usepackage{latexsym}
\usepackage{amsthm}
\usepackage{amscd}
\usepackage{epsf}
\usepackage{graphicx}
\usepackage{makeidx}
\usepackage{enumerate}
\usepackage{mathdots}

\newtheorem{theorem}{Theorem}[section]

\newtheorem{lemma}[theorem]{Lemma}

\newtheorem{definition}[theorem]{Definition}

\newcommand{\bsa}{\boldsymbol{a}}
\newcommand{\bsb}{\boldsymbol{b}}
\newcommand{\bsc}{\boldsymbol{c}}

\newcommand{\bsx}{\boldsymbol{x}}
\newcommand{\bsy}{\boldsymbol{y}}

\newcommand{\bszero}{\boldsymbol{0}}

\newcommand{\Fcal}{\mathcal{F}}

\newcommand{\Jcal}{\mathcal{J}}
\newcommand{\Kcal}{\mathcal{K}}

\newcommand{\Pcal}{\mathcal{P}}

\def\bN{\mathbb N}
\def\bZ{{\mathbb Z}}

\def\bsx{{\boldsymbol{x}}}
\def\bsy{{\boldsymbol{y}}}
\def\bsa{{\boldsymbol{a}}}

\def\pr{{\mathrm{pr}}}

\def\id{{{\mathrm{id}}}}

\newcommand{\Lset}[1][m]{\mathcal{L}_{#1}}
\newcommand{\Uset}[1][m]{\mathcal{U}_{#1}}

\newcommand{\Ftwo}{\mathbb{F}_2}
\newcommand{\FtwoMat}{\mathbb{F}_2^{m \times m}}

\newenvironment{enuroman}{\begin{enumerate}[\normalfont (i)]}{\end{enumerate}}

\begin{document}

\title{
Characterization of matrices $B$
such that $(I,B,B^2)$ generates a digital net 
with $t$-value zero\thanks{H.K.\ is supported by JST CREST,
M.M.\ by JST CREST and JSPS/MEXT Grant-in-Aid No.\ 26310211, 15K13460,
and
K.S.\ is supported by JST CREST and JSPS Grant-in-Aid for JSPS Fellows (No.\ 17J00466).}
}
\date{\today}
\author{
Hiroki Kajiura\thanks{Graduate School of Science, Hiroshima University.
1-3-1 Kagamiyama, Higashi-Hiroshima, 739-8526, Japan.
e-mail: hikajiura@hiroshima-u.ac.jp},
Makoto Matsumoto\thanks{Graduate School of Science, Hiroshima University.
1-3-1 Kagamiyama, Higashi-Hiroshima, 739-8526, Japan.
e-mail: m-mat@math.sci.hiroshima-u.ac.jp},
Kosuke Suzuki
\thanks{Graduate School of Science, Hiroshima University. 1-3-1 Kagamiyama, Higashi-Hiroshima, 739-8526, Japan. JSPS Research Fellow. e-mail: kosuke-suzuki@hiroshima-u.ac.jp}
}
\maketitle

\begin{abstract}
We study $3$-dimensional digital nets over $\Ftwo$ generated by matrices $(I,B,B^2)$
where $I$ is the identity matrix and $B$ is a square matrix.
We give a characterization of $B$ for which 
the $t$-value of the digital net is $0$.
As a corollary, we prove that such $B$ satisfies $B^3=I$.
\end{abstract}

\section{Introduction and main result}
Let $\Ftwo = \{0,1\}$ be the field of two elements, 
$m \geq 1$ be a positive integer,
and $\FtwoMat$ be the set of $m \times m$ matrices over $\Ftwo$.
For $C_1, \dots, C_s \in \FtwoMat$,
the digital net generated by $(C_1, \dots, C_s)$ is a point set in $[0,1)^s$
defined as follows.
For $0 \leq l <2^m$, we denote the $2$-adic expansion of $l$
by $l=\iota_0 + \iota_12 + \dots + \iota_{m-1}2^{m-1}$
with $\iota_0, \dots, \iota_{m-1}\in \Ftwo$.
We define $\bsy_{l,j} \in \Ftwo^m$ for $1 \leq j \leq s$ as
\[
\bsy_{l,j} := C_j (\iota_0, \dots, \iota_{m-1})^\top \in \Ftwo^m.
\]
Then we obtain the $l$-th point
\begin{equation}\label{eq:pt-generation}
\bsx_l := (\phi(\bsy_{l,1}), \dots, \phi(\bsy_{l,s}))
\end{equation}
where $\phi \colon \Ftwo^m \to [0,1)$ is defined as
\[
\phi((y_1, \dots, y_{m})^\top) := \frac{y_{1}}{2}+\frac{y_{2}}{2^2}+ \dots + \frac{y_{m}}{2^m}.
\]
The digital net generated by $(C_1, \dots, C_s)$
is the point set $\{\bsx_0, \dots, \bsx_{2^m-1}\} \subset [0,1)^s$.
Digital nets are introduced by Niederreiter
and have been widely used to generate point sets in Quasi-Monte Carlo
(QMC) theory,
see \cite{Niederreiter1992rng} for details.

A popular criterion of the uniformity of digital nets is the $t$-value.
Let $m \geq 1$, $0 \leq t \leq m$, and $s \geq 1$ be integers.
A point set $P = \{\bsx_0, \dots, \bsx_{2^m - 1}\} \subset [0,1)^s$
is called a $(t,m,s)$-net over $\Ftwo$
if, for all nonnegative integers $d_1, \dots , d_s$ with $d_1 + \dots + d_s = m - t$,
the elementary intervals
$\prod_{i=1}^{s}{\left[a_i/2^{d_i}, (a_i +1)/2^{d_i} \right)}$
contain exactly $2^t$ points for all choices of
$0 \leq a_i < 2^{d_i}$ with $a_i \in \bZ$ for $1 \leq i \leq s$.
In this paper we study $t$-values of specific $3$-dimensional digital nets over $\Ftwo$.
Small value of $t$ is preferable for QMC integration
\cite{Niederreiter1992rng}.
 
To state our main result, we introduce our notation.
Let $I_m$ be the $m \times m$ identity matrix.
Let $J_m$ be the $m \times m$ anti-diagonal matrix whose anti-diagonal entries are all 1,
and $P_m$ be the $m \times m$ upper-triangular Pascal matrix, i.e.,
\[
J_m =
\begin{pmatrix}
0 &  & 1\\
& \iddots \\
1 & &  0
\end{pmatrix},
\qquad
P_m = \left(\binom{j-1}{i-1}\right)_{i,j=1}^m =
\begin{pmatrix}
\binom{0}{0} & \binom{1}{0} & \dots & \binom{m-1}{0}\\
 & \binom{1}{1} & & \vdots   \\
 & & \ddots & \vdots  \\
 & &  & \binom{m-1}{m-1}
\end{pmatrix},
\]
which are considered in modulo $2$.
If there is no confusion, we omit the subscripts
and simply write as $I$, $J$, and $P$.
Let $\Lset$ (resp.\ $\Uset$) be the set of $m \times m$ lower- (resp.\ upper-) triangular matrices
over $\Ftwo$.
Note that $\Lset \cap \Uset = \{I\}$ holds.
For matrices $C_1, \dots, C_s \in \FtwoMat$,
$t(C_1, \dots, C_s)$ denotes the $t$-value of the digital net generated by $(C_1, \dots, C_s)$. 

Now we are ready to state our main result.
\begin{theorem} \label{thm:IBB2-structure}
Let $m \geq 1$ be an integer and $B \in \FtwoMat$.
Then the following are equivalent.
\begin{enuroman}
\item \label{eq:IBB2-structure-equiv1}
$t(I,B, B^2)=0$.
\item \label{eq:IBB2-structure-equiv2}
There exists $L \in \Lset$
such that $B=LPJL^{-1}$.
\end{enuroman}
Moreover, if one of the above holds, then we have $B^3 = I$.
\end{theorem} 
Note that for digital nets over $\Ftwo$, $t(C_1, \dots, C_s)=0$ is achievable
if and only if $s\leq 3$ 
(see \cite[Corollary~4.21]{Niederreiter1992rng} or \cite{Faure1982dds}).
Thus, the above theorem shows that this extreme $s=3$
can be realized in the special form $t(I,B,B^2)$.

\paragraph{Background.}
Our original motivation is 
to find a periodic sequence for Markov Chain Quasi-Monte Carlo (MCQMC) method.
Let us recall the rough idea.
Let $x_1, x_2, \dots$ be a sequence of points in $[0,1)$. For an integer $s \geq 1$,
we define 
\begin{equation}
\bar\bsx^{(s)}_i = (x_{i}, x_{i+1}, \dots x_{i+s-1}) \in [0,1)^s, \label{eq:overlapped sequence}
\end{equation}
where they are made up of overlapping consecutive $s$-tuples from the sequence.
The sequence is said to be completely uniformly distributed (CUD)
if $\bar\bsx^{(s)}_1, \bar\bsx^{(s)}_2, \dots$ is uniformly distributed in $[0,1)^s$ for all $s \geq 1$.

We do not explain on MCQMC method,
but it is shown that CUD sequence can be used instead of uniformly
i.i.d.\ uniform random numbers in $[0,1)$. 
Markov Chain Monte Carlo (MCMC) with the driving sequence being CUD is consistent to the original MCMC,
see \cite{Chen2011cmc}. 

Constructions for CUD points given in \cite{Levin1999dec} are not convenient to implement.
Instead, it was suggested by Tribble \cite{Tribble2007mcm} to use multiple congruential
generators and linear feedback shift registers.
Chen et.\ al.\ \cite{Chen2012nim} considered a periodic sequence 
$x_1, x_2, \dots$ with period $p$, the $s$-dimensional point set
\begin{equation}
S_s:=\{\bar\bsx^{(s)}_i = (x_{i}, x_{i+1}, \dots x_{i+s-1}) \in [0,1)^s
\mid i=1,\ldots, p\}
\end{equation}
whose cardinality is $p$ as a multi set. It is expected to work well for MCQMC
if $S_s$ is hyperuniform for every $s$.
Assume that $S_s\cup \{0\}$ is a $\Ftwo$-sub vector space (this
condition is necessary to compute $t$-value in a practical time) of,
say, dimension $m$. Here, each $x_i$ is assumed to be identified with 
an element in $\Ftwo^m$ through $\phi$. Let $V$ be this vector space
$S_s\cup \{0\}$. 
We further require that $S_1$ is a $(0,m,1)$-net. Then, the projection to the 
first component $\pr_1\colon V \to \Ftwo^m$ is linearly isomorphic.
This implies that the second projection $\pr_2\colon V \to \Ftwo^m$ is also isomorphic
since the images of them are the same.
Thus $\pr_2 \circ \pr_1^{-1}$ is also isomorphic.
This means that there is a fixed $B\in \FtwoMat$ such that
\[
x_{i+1} =Bx_i
\]
holds for $i=1,2,\ldots$. Moreover, since we have assumed that $S_s\cup\{0\}$
is a $m$-dimensional vector space, $x_i$ must take all non zero values
once for $1\leq i \leq p-1$. This is equivalent that $B$ is primitive
(i.e., the multiplicative order of $B$ is $2^{m}-1$ and $p=2^m-1$).
This type of pseudorandom number generator is well studied, such as
combined Tausworthe generators, see L'Ecuyer et.\ al.\ \cite{LEcuyer1999qmc}.
Under our assumptions, we observe that 
the set
\[
\{\bar\bsx^{(s)}_i \mid 0 \leq i < 2^m-1\} \cup \{\bszero\}
\]
is the digital net generated by $(I,B,B^2, \dots,B^{s-1})$, as a set.

Our original interest is to obtain such a maximal periodic $B$ 
with small $t$-value for wide $s$, to generate a pseudo-CUD sequence.
For example, $t=0$ might be possible for $s=2$, which is the 
theoretical bound stated above (below Theorem~\ref{thm:IBB2-structure}).
However, an exhaustive search for matrices $B$ with $t(I,B,B^2)=0$ for
$m\leq 5$ resulted non-primitive $B$. Actually, we obtained a negative result
Theorem~\ref{thm:IBB2-structure}:
For $s=3$ and $m \geq 3$, the digital net generated by $(I,B,B^2)$ is a $(0,m,3)$-net
only if $B^3=I$.
Thus there is no $B \in \FtwoMat$ satisfying our assumptions for $m\geq 3$.
Hence we conclude that our construction of $\Ftwo$-linear generator with maximal period
is not optimal with respect to the $t$-value for $s=3$. We need to
consider some looser condition, such as considered in \cite{Chen2012nim}.

\section{Preliminaries}
We first recall results for $t$-value of digital nets.
It is known that $t$-value of digital nets is related to the linear independence of column vectors of generating matrices.
\begin{lemma}[{\cite[Theorem~4.52]{Dick2010dna}}]\label{lem:tval-linear-indpendence}
Let $C_1, \dots, C_s \in \FtwoMat$
and denote by $\bsc_i^j$ the $j$-th row of $C_i$.
Assume that,
for all choices of nonnegative integers $d_1, \dots , d_s$ with $d_1 + \dots + d_s = m - t$,
$m-t$ vectors $\{\bsc_i^j \mid 1 \leq j \leq d_i \}$ are linear independent.
Then the digital net generated by $(C_1, \dots, C_s)$ is a $(t,m,s)$-net over $\Ftwo$.
\end{lemma}

\begin{lemma}\label{lem:tval-invariant}
Let $C_1, \dots, C_s \in \FtwoMat$ and $L_1, \dots, L_s \in \Lset$.
Let $G \in \FtwoMat$ be non-singular.
Then we have $t(C_1, \cdots, C_s)=t(L_1C_1G, \cdots, L_sC_sG)$.
\end{lemma}

\begin{proof}
Since $G$ is non-singular, $(L_1C_1, \cdots, L_sC_s)$ and $(L_1C_1G, \cdots, L_sC_sG)$
generate the same digital net (as set) and hence we have
$t(L_1C_1, \cdots, L_sC_s) = t(L_1C_1G, \cdots, L_sC_sG)$.
Further, since $L_1, \dots, L_s \in \Lset$,
multiplying them from left does not change the linear independence
appearing in Lemma~\ref{lem:tval-linear-indpendence}. Thus it does not change the $t$-value,
i.e., $t(C_1, \cdots, C_s) = t(L_1C_1, \cdots, L_sC_s)$.
\end{proof}

In the rest of this section,
we give explicit $B$ where 
the digital net generated by $(I,B,B^2)$ is a $(0,m,3)$-net over $\Ftwo$. 
To this end, we introduce the notion of $(t,s)$-sequence.
\begin{definition}
Let $t \geq 0$ and $s \geq 1$ be integers.
A sequence $\bsx_0, \bsx_1, \dots$ of points in $[0,1)^s$
is said to be a $(t,s)$-sequence over $\Ftwo$
if, for all integers $k \geq 0$ and $m>t$, the point set
$\{\bsx_n \mid k2^m \leq n < (k+1)2^m\}$ forms a $(t,m,s)$-net over $\Ftwo$.
\end{definition}
There are many known explicit constructions of digital nets with low $t$-value.
Among them we introduce the Faure sequence \cite{Faure1982dds}.
The Faure sequence over $\Ftwo$ is a $(0,2)$-sequence
where the $l$-th point $\bsx_l \in [0,1)^2$ is generated 
as in \eqref{eq:pt-generation} by matrices $(I_m, P_m)$
(note that it gives the same $\bsx_l$ even if $m$ is different),
see, for example, \cite[Section~8.1]{Dick2010dna}.

From $(t,s)$-sequence, we can generate $(t,m,s+1)$-net \cite[Lemma~4.22]{Niederreiter1992rng}.
\begin{lemma}\label{lem:seq-to-net}
Let $\{\bsx_i\}_{i \geq 0}$ be $(t,s)$-sequence over $\Ftwo$.
Then $\{(\bsx_i, i 2^{-m})\}_{i=0}^{2^m -1}$ is a $(t,m,s+1)$-net over $\Ftwo$.
\end{lemma}

When $\{\bsx_0, \dots, \bsx_{2^{m}-1}\}$ is the first $2^m$ points of the Faure sequence over $\Ftwo$,
which is the digital net generated by $(I_m, P_m)$,
the $3$-dimensional point set $\{(\bsx_i, i 2^{-m})\}_{i=0}^{2^m -1}$
is found to be a digital net generated by $(I_m, P_m, J_m)$.
Thus it follows from Lemma~\ref{lem:seq-to-net} that
\begin{equation}\label{eq:0m3net-example}
t(I_m, P_m, J_m) = 0.
\end{equation}

We move on to the property of the matrix $PJ$.
\begin{lemma}\label{lem:PJPJPJ}
For any positive integer $m$,
we have
\[
P^2 = J^2 = (PJ)^3 = I \quad \text{in $\FtwoMat$}.
\]
\end{lemma}

\begin{proof}
It is clear to check $J^2=I$.
We now prove $P^2=I$ in $\FtwoMat$. 
Let $k$ be a field and $k(x)$ a field of rational functions. 
Define two ring endmorphisms:
\[
\Pcal \colon k(x) \to k(x); \quad x \mapsto (1-x), \qquad
\Kcal \colon k(x) \to k(x); \quad x \mapsto x^{-1}.
\]
Define also a $k$-linear map
\[
\Jcal \colon k(x) \to k(x); \quad f(x) \mapsto x^{m-1} \cdot \Kcal(f(x)).
\]

Let $V_m := \langle 1, x, \dots, x^{m-1} \rangle$ be a $k$-linear subspace of $k(x)$.
Then the restriction of $\Pcal$ and $\Jcal$ on $V_m$
are $k$-linear endomorphisms.
We find that the representation matrix of $\Pcal$ restricted to $V_m$
has coefficients of $P'_m$ defined as
\[
P'
= P'_m
:= \left((-1)^{i-1} \binom{j-1}{i-1}\right)_{i,j=1}^m
\]
Note that $P=P'$ in modulo $2$.
It is clear that the representation matrix of $\Jcal$ restricted to $V_m$ is $J_m$.
We will show equalities between matrices
via showing corresponding equalities between $k$-linear endomorphisms on $k(x)$.

For two $k$-ring endomorphisms $\Fcal_1, \Fcal_2 \colon k(x) \to k(x)$,
$\Fcal_1 = \Fcal_2$ holds if and only if $\Fcal_1(x) = \Fcal_2(x)$ holds,
since $k(x)$ is generated by $x$ as a ring 
(to be precise we need to consider $x^{-1}$ as well,
but the inverse element is preserved by a ring homomorphism).
From this property we have
\begin{equation}\label{eq:P2K2PKPKPK}
\Pcal^2 = \Kcal^2 = \Pcal \Kcal \Pcal \Kcal \Pcal \Kcal = \id_{k(x)},
\end{equation}
since all of them map $x$ to itself.
Thus, by restricting $\Pcal^2 = \id_{k(x)}$ on $V_m$, we have $P'^2=I$.
Hence $P^2=I$ in $\FtwoMat$.

We now show $(PJ)^3=I$ in $\FtwoMat$.
For $a \in k(x)$, we define the multiplication map
\[
(a \times) \colon k(x) \to k(x), \qquad f(x) \mapsto af(x).
\]
Then
\[
\Pcal \circ (a \times) = \Pcal(a) \cdot \Pcal \quad and \quad
\Kcal \circ (a \times) = \Kcal(a) \cdot \Kcal
\]
hold. Using this property and \eqref{eq:P2K2PKPKPK}, we have
\begin{align*}
\Pcal \Jcal \Pcal \Jcal \Pcal \Jcal
&= \Pcal\circ(x^{m-1}\times)\circ\Kcal\Pcal\circ(x^{m-1}\times)\circ\Kcal\Pcal\circ(x^{m-1}\times)\circ\Kcal \\
&= \Pcal(x^{m-1}) \cdot \Pcal\Kcal\Pcal(x^{m-1}) \cdot \Pcal\Kcal\Pcal\Kcal\Pcal(x^{m-1}) \cdot \Pcal\Kcal\Pcal\Kcal\Pcal\Kcal\\
&= (-1)^{m-1} \id_{k(x)}.
\end{align*}
By restricting above to $V_m$, whenever $k$ has characteristic $2$ we have
\[
(PJ)^3=I,
\]
as we wanted.
\end{proof}

We now show that the matrix $PJ$ is what we want.
\begin{lemma}\label{lem:explicit-construction}
For any positive integer $m$, we have
\[
t(I_m,P_mJ_m,(P_mJ_m)^2) = 0.
\]
\end{lemma}

\begin{proof}
Lemma~\ref{lem:PJPJPJ} implies $(PJP)^{-1}=JPJ$. Further $JPJ \in \Lset$ holds.
Hence by Lemma~\ref{lem:tval-invariant} with $(L_1,L_2,L_3) = (J,I,JPJ)$ and $G=J$ we have
\[
t(I,PJ,(PJ)^2) = t(J,P,I) = 0,
\]
where the last equality follows from \eqref{eq:0m3net-example}.
\end{proof}

\section{Proof of Theorem~\ref{thm:IBB2-structure}}

To prove Theorem~\ref{thm:IBB2-structure},
we need the following lemmas which will be shown in Section~\ref{sec:proof}.
\begin{lemma}\label{lem:IB-0m2net-structure}
Let $B \in \FtwoMat$. Then the following are equivalent.
\begin{enuroman}
\item \label{eq:IB-structure-equiv1}
$t(I,B)=0$.
\item \label{eq:IB-structure-equiv2}
There exist $L_1, L_2 \in \Lset$
such that $B = L_1 J L_2$.
\end{enuroman}
\end{lemma}

\begin{lemma}\label{lem:0m3net-structure}
Let $A,B,C,C' \in \FtwoMat$. 
Suppose that $t(A,B,C) = t(A,B,C') = 0$.
Then there exists $L \in \Lset$
such that $LC = C'$.
\end{lemma}
Assuming the above lemmas, we show the main theorem.

\begin{proof}[Proof of Theorem~\ref{thm:IBB2-structure}]
First we assume \eqref{eq:IBB2-structure-equiv2}.
By Lemma~\ref{lem:tval-invariant} with $(L_1,L_2,L_3) = (L^{-1},L^{-1},L^{-1})$ and $G=L$
we have
\begin{align*}
t(I,B,B^2)
&= t(I, LPJL^{-1}, LPJPJL^{-1})\\
&= t(I, PJ, PJPJ) = 0.
\end{align*}
Here the last equality follows from Lemma~\ref{lem:explicit-construction}.
Hence \eqref{eq:IBB2-structure-equiv1} follows.

We now assume \eqref{eq:IBB2-structure-equiv1}.
By Lemma~\ref{lem:IB-0m2net-structure},
there exists $L_1, L_2 \in \Lset$
such that $B = L_1 J L_2$.
Then
by Lemma~\ref{lem:tval-invariant} with $(L_1,L_2,L_3) = (L_2^{-1},L_1,L_1)$ and $G=L_2$
we have\[
t(I ,J, JL_2L_1J)
= t(I, L_1JL_2, (L_1JL_2)^2)
= t(I, B, B^2)
= 0.
\]
On the other hand, from \eqref{eq:0m3net-example}
we have $t(I, J, P) =0$.
Hence it follows from  Lemma~\ref{lem:0m3net-structure}
that there exists $L_3 \in \Lset$
such that $L_3JL_2L_1J = P$
and thus $L_3  = P(JL_2L_1J)^{-1}$.
Since $L_3 \in \Lset$ and
$P(JL_2L_1J)^{-1} \in \Uset$ hold,
both are equal to $I$.
Thus $L_3 = I$ and $JL_2L_1J = P$ hold,
and the latter implies $L_2 = JPJL_1^{-1}$.
Hence
$B =  L_1JL_2 = L_1PJL_1^{-1}$,
which shows \eqref{eq:IBB2-structure-equiv2}.

We now assume that one of them holds
(and thus \eqref{eq:IBB2-structure-equiv2} holds).
Then there exist $L \in \Lset$ such that $B=LPJL^{-1}$.
Hence we have
\[
B^3 = (LPJL^{-1})^3 = L(PJ)^3 L^{-1} = LL^{-1} = I,
\]
where the the third equality follows from Lemma~\ref{lem:PJPJPJ}.
\end{proof}

\section{Proofs of lemmas}\label{sec:proof}
\subsection{Proof of Lemma~\ref{lem:IB-0m2net-structure}}
\begin{proof}[Proof of Lemma~\ref{lem:IB-0m2net-structure}]
First we assume \eqref{eq:IB-structure-equiv2}.
By Lemma~\ref{lem:tval-invariant} with $(L_1,L_2) = (L_2^{-1},L_1^{-1})$ and $G=L_2^{-1}$
we have
\[
t(I,B)
= t(I,  L_1 J L_2)
= t(I, J) = 0.
\]
and thus \eqref{eq:IB-structure-equiv1} follows.

We now assume \eqref{eq:IBB2-structure-equiv1}.
From this we have $t(J, BJ)= t(I, B) =0$.
From $t(J, BJ)=0$,
we can show that all of the leading principal minor matrices of $BJ$ are non-singular.
Hence there exist $L \in \Lset$ and $U \in \Uset$ such that $LBJU = I$.
Thus we have
\[
B
= L^{-1}U^{-1}J
= L^{-1}J^2U^{-1}J
= L^{-1}J(JU^{-1}J).
\]
This shows \eqref{eq:IBB2-structure-equiv1} since $JU^{-1}J \in \Lset$.
\end{proof}

\subsection{Proof of Lemma~\ref{lem:0m3net-structure}}
Here we prove two lemmas to show Lemma~\ref{lem:0m3net-structure}.

Let us denote
\[
A =
\begin{pmatrix}
\bsa_1\\
\bsa_2\\
\vdots\\
\bsa_m
\end{pmatrix},
\qquad
B =
\begin{pmatrix}
\bsb_1\\
\bsb_2\\
\vdots\\
\bsb_m
\end{pmatrix},
\qquad
C =
\begin{pmatrix}
\bsc_1\\
\bsc_2\\
\vdots\\
\bsc_m
\end{pmatrix},
\qquad
C' =
\begin{pmatrix}
\bsc'_1\\
\bsc'_2\\
\vdots\\
\bsc'_m
\end{pmatrix}.
\]

\begin{lemma}\label{lem:dim-cardinality-bsc}
Let $A,B,C \in \FtwoMat$ and assume that $t(A,B,C)=0$.
For $i,j \in \bN$ with $i+j \leq m-1$,
we define a subspace $V_{i,j}$ of $\Ftwo^{1 \times m}$ as
\[
V_{i,j} := \langle \bsa_1, \dots, \bsa_i, \bsb_1, \dots, \bsb_{m-i-j-1}, \bsc_1, \dots, \bsc_j \rangle.
\]
Let $1 \leq k \leq m-j$ and $0 \leq  i_1 < \cdots < i_k \leq m-1-j$ be integers.
Then the following holds true.
\begin{align}
\dim \bigcap_{l=1}^k V_{i_l, j} = m-k, \label{eq:cardinality-cap} \\
\left| \bigcap_{0 \leq i \leq m-j-1} V_{i,j}^c \right| = 2^j \label{eq:cardinality-bsc2}.
\end{align}
\end{lemma}

\begin{proof}
First we show \eqref{eq:cardinality-cap} by induction on $k$.
The assumption that $t(A,B,C) = 0$ implies that
$\dim V_{i,j} = m-1$ for all $i$ and $j$. This shows the lemma for $k=1$.
We now assume the lemma for $k-1$ and show for $k$.
Fix $0 \leq  i_1 < \cdots < i_k \leq m-1-j$ and 
let $U := \bigcap_{l=2}^k V_{i_l, j}$.
It follows from $t(A,B,C)=0$ that
$\bsa_{i_1+1} \notin V_{i_1, j}$.
Combining this with $\bsa_{i_1+1} \in U$, we have
\[
m = 1 + \dim V_{i_1, j} \leq \dim (U + V_{i_1, j}) \leq m,
\]
which shows $\dim (U + V_{i_1, j}) = m$.
Further we have $\dim U = m-k+1$ by induction assumption.
Thus we have
\begin{align*}
\dim (U \cap V_{i_1, j})
&= \dim U + \dim V_{i_1,j} - \dim (U + V_{i_1, j})\\
&= (m-k+1) + (m-1) - m\\
&= m-k.
\end{align*}
This shows the lemma for $k$.

Now we show \eqref{eq:cardinality-bsc2}.
By \eqref{eq:cardinality-cap} and the inclusion-exclusion principle, we have
\begin{align*}
\left| \bigcap_{0 \leq i \leq m-j-1} V_{i,j}^c \right|
&= 
\left| \Ftwo^{1 \times m} \right|
- \sum_{\emptyset \neq S \subset \{0, 1, \dots, m-j-1\}}(-1)^{|S|} \left| \bigcap_{i \in S} V_{i,j} \right|\\
&= 2^m - \sum_{\emptyset \neq S \subset \{0, 1, \dots, m-j-1\}}(-1)^{|S|} 2^{m-|S|}\\
&= 2^m - \sum_{k=1}^{m-j}(-1)^k 2^{m-k} \sum_{\emptyset \neq S \subset \{0, \dots, m-j-1\}, |S|=k} 1\\
&= 2^m - \sum_{k=1}^{m-j}(-1)^k 2^{m-k} \binom{m-j}{k} \\
&= 2^j \sum_{k=0}^{m-j}(-1)^k 2^{m-j-k} \binom{m-j}{k} \\
&= 2^j (2-1)^{m-j} \\
&= 2^j.
\end{align*}
This shows \eqref{eq:cardinality-bsc2}.
\end{proof}

\begin{lemma}\label{lem:0m3-structure}
Under the assumption and notation of Lemma~\ref{lem:dim-cardinality-bsc},
we further assume that $C' \in \FtwoMat$ and $t(A,B,C')=0$.
For $i,j \in \bN$ with $i+j \leq m-1$ we define a subspace $W_{i,j}$ of $\Ftwo^{1 \times m}$ as
\[
W_{i,j} := \langle \bsa_1, \dots, \bsa_i, \bsb_1, \dots, \bsb_{m-i-j-1}, \bsc'_1, \dots, \bsc'_j \rangle.
\]
Then the following holds true.
\begin{enuroman}
\item
$\bsc'_{j} \in \bsc_{j} + \langle \bsc_1, \dots, \bsc_{j-1} \rangle$ for $j \geq 1$ \label{eq:0m3-structure-1},
\item
$\langle \bsc_1, \dots, \bsc_j \rangle = \langle \bsc'_1, \dots, \bsc'_j \rangle$ for $j \geq 1$ \label{eq:0m3-structure-12}
\item $V_{i,j} = W_{i,j}$ for all $i$ and $j$, \label{eq:0m3-structure-2}
\end{enuroman}
\end{lemma}

\begin{proof}
We show the lemma by induction on $j$.
When $j=0$, trivially $V_{i,0} = W_{i,0}$.
We now assume the claim for $j$ and show for $j+1$.
It follows from $t(A,B,C) = 0$ that $\bsc_{j+1} \notin V_{i,j}$ for all $i$.
Further we have $\langle \bsc_1, \dots, \bsc_j \rangle \subset V_{i,j}$ for all $i$.
Hence
\begin{equation}\label{eq:inclusion-infact-equal}
\bsc_{j+1} + \langle \bsc_1, \dots, \bsc_j \rangle \subset \bigcap_{0 \leq i \leq m-j-1} V_{i,j}^c.
\end{equation}
The cardinality of the left hand side is $2^j$,
and that of the right hand side is also $2^j$ from Lemma~\ref{lem:dim-cardinality-bsc}.
Thus we have
\[
\bsc_{j+1} + \langle \bsc_1, \dots, \bsc_j \rangle = \bigcap_{0 \leq i \leq m-j-1} V_{i,j}^c.
\]
In the same way, it holds that 
\[
\bsc'_{j+1} + \langle \bsc'_1, \dots, \bsc'_j \rangle
= \bigcap_{0 \leq i \leq m-j-1} W_{i,j}^c
= \bigcap_{0 \leq i \leq m-j-1} V_{i,j}^c.
\]
where the last equality follows from induction assumption.
Hence we have
\[
\bsc_{j+1} + \langle \bsc_1, \dots, \bsc_j \rangle
= \bsc'_{j+1} + \langle \bsc'_1, \dots, \bsc'_j \rangle.
\]
In particular, using the induction assumption of \eqref{eq:0m3-structure-12}, we have 
\[
\langle \bsc_1, \dots, \bsc_{j+1} \rangle = \langle \bsc'_1, \dots, \bsc'_{j+1} \rangle
\qquad \text{and} \qquad
\bsc'_{j+1} \in \bsc_{j+1} + \langle \bsc_1, \dots, \bsc_j \rangle.
\]
This shows \eqref{eq:0m3-structure-1} and \eqref{eq:0m3-structure-12} for $j+1$.
This implies
\begin{align*}
V_{i,j+1}
&= \langle \bsa_1, \dots, \bsa_i, \bsb_1, \dots, \bsb_{m-i-j-2}, \bsc_1, \dots, \bsc_{j+1} \rangle\\
&= \langle \bsa_1, \dots, \bsa_i, \bsb_1, \dots, \bsb_{m-i-j-2}, \bsc'_1, \dots, \bsc'_{j+1} \rangle
= W_{i,j+1},
\end{align*}
which shows \eqref{eq:0m3-structure-2} for $j+1$.
\end{proof}

Now Lemma~\ref{lem:0m3net-structure} is easy to show:
Lemma~\ref{lem:0m3-structure} \eqref{eq:0m3-structure-1}
directly implies that there exists $L \in \Lset$ such that  $LC = C'$.

\bibliographystyle{plain}
\bibliography{t0ord3}

\end{document}